\documentclass[12pt,twoside]{amsart}
\usepackage{amsthm,amsfonts,amssymb,amscd,euscript}

\usepackage[matrix,arrow]{xy}

\author{Florin Ambro} 
\address{Institute of Mathematics ``Simion Stoilow'' of the Romanian
Academy\\
P.O. BOX 1-764, RO-014700 Bucharest\\ 
Romania.}
\email{florin.ambro@imar.ro}

\newcommand{\Z}{{\mathbb Z}}

\newcommand{\R}{{\mathbb R}}

\newcommand{\bP}{{\mathbb P}} 

\newcommand{\Conv}{\operatorname{Conv}}

\newcommand{\emb}{\operatorname{emb}}

\newcommand{\Int}{\operatorname{int}}

\newcommand{\length}{\operatorname{length}}

\newcommand{\mld}{\operatorname{mld}}

\newcommand{\vol}{\operatorname{vol}}

\theoremstyle{plain}
\newtheorem{thm}{Theorem}[section]

\newtheorem{lem}[thm]{Lemma}

\newtheorem{prop}[thm]{Proposition}

\theoremstyle{definition}

\newtheorem{exmp}[thm]{Example}

\newtheorem{ack}{Acknowledgments}   

\theoremstyle{remark}

\begin{document}

\bibliographystyle{amsalpha+}
\title{An explicit bound for complements on toric varieties}
\maketitle

\begin{abstract} 
We give an effective upper bound for the index of klt complements on toric Fano varieties.
\end{abstract} 





\section*{Introduction}


We are interested in the following fundamental theorem of Caucher Birkar~\cite{Birk21}:
for a Fano variety $X$ of fixed dimension, the global minimal log discrepancy $\mld X$ is bounded away 
from zero if and only if there exists a klt complement of index bounded above. 

The converse 
implication is straightforward: if $B\ge 0$, $n(K+B)\sim 0$ and $(X,B)$ is klt, then $\mld(X,B)\ge \frac{1}{n}$.
Since $B$ is effective, we obtain $\mld X\ge \mld(X,B) \ge \frac{1}{n}$. The direct implication was 
proved by C. Birkar via a very sophisticated argument: assuming $\mld X\ge \epsilon$, 
one first constructs an lc complement of bounded index, which is later upgraded to a klt
complement of bounded index $n$, using lower bounds for the $\alpha$-invariant of $-K$ and
birational boundedness of log canonical models of small volume. A natural question is 
to find effective (possibly sharp) upper bounds for $n$ in terms of $\epsilon$.
Finding sharp upper bounds may lead to interesting examples, but also to a simpler, more 
direct argument of C. Birkar's theorem.

Our main result is an effective upper bound in the case of toric Fano varieties.
Let $\epsilon\in (0,1]$. For $n\in \Z_{\ge 1}$, define recursively $\lambda(1,\epsilon)=\frac{2}{\epsilon}$
and $\lambda(n,\epsilon)=\frac{n(n+1)}{\epsilon}+\lambda(n-1,\frac{\epsilon^2}{n(n+1)})$ for $n\ge 2$.
For fixed $n$, $\lambda(n,\epsilon)$ is a polynomial with non-negative integer coefficients in $\epsilon^{-1}$,
of degree $2^{n-1}$. In particular, $\lambda(n,t)$ is decreasing in $t$. Moreover, the following inequality 
holds by induction:
$$
\lambda (n,\epsilon)\le \sum_{i=0}^{n-1} (\frac{n(n+1)}{\epsilon})^{2^i}.
$$

\begin{thm}\label{kF} 
	Let $(X,B)$ be a toric weak log Fano variety, with $B$ effective invariant. 
	Suppose $\dim X=d$ and $\mld(X,B)\ge \epsilon>0$. Then 
	$
	\sqrt[d]{(-K-B)^d} > \frac{2}{\lambda(d,\epsilon)}
	$
	and for every integer $n\ge \lambda(d,\epsilon)$, the following properties hold:
	\begin{itemize}
		\item The linear system of $\R$-Weil divisors $|-nK-nB|$ defines a birational map.
		\item Let $D\in |-nK-nB|$ be a general member. Then $n(K+B+\frac{D}{n})\sim 0$, $\mld(X,B+\frac{D}{n})\ge \frac{1}{n}$.
	\end{itemize}
\end{thm}

The proof of Theorem~\ref{kF} is combinatorial, with the following geometric interpretation.
If $\Sigma_X$ is the complement of the open torus inside $X$, then $K+\Sigma_X\sim 0$
is a log canonical complement of index one. We construct a rational map $\phi \colon X\dashrightarrow \bP^1$
induced by a pencil $\Lambda_1 \subset |-n_1K-n_1B|$, with $n_1$ bounded, such that
for a general member $D_1\in \Lambda_1$, $K+B+\frac{1}{n_1}D_1\sim_{n_1}0$ 
has lc singularities and all its lc centers dominate $\bP^1$ via $\phi$.
Let $\mu\colon X'\to X$ be the normalization of the graph of $\phi$, let $\mu^*(K+B)=K_{X'}+B_{X'}$ be
the log pullback. Let $F$ be the general fiber of the fibration $\phi'\colon X'\to \bP^1$, let
$(K_{X'}+B_{X'})|_F=K_F+B_F$. Then $\dim F=\dim X-1$, and the claim lifts by induction
from $F$ to $X$. This is true provided $B_F$ is effective. If not, 
the boundedness of $n_1$ allows to make $B_F$ effective after replacing it with a convex combination with the toric 
lc complement $\Sigma_F\in |-K_F|$, at the expense of scaling down $\mld(F,B_F)$ by a bounded factor.

\begin{ack} 
I would like to thank Caucher Birkar, Yu Zou and the anonymous referee for comments and corrections.
\end{ack}


\section{Preliminary}


We use standard notation on toric varieties and convex sets, cf.~\cite{Oda88}.


\subsection{Toric weak log Fano varieties}


Let $(X^d,B)$ be a proper toric variety $X$ endowed with invariant boundary $B$ 
such that $-K-B$ is $\R$-free and $\mld(X,B)\ge \epsilon>0$. Since $X$ is proper,
it follows that $-K-B$ is big. 

Let $X=T_N\emb(\Delta)$. Let 
$M=N^*$ be the dual lattice. Let $\square_{-K-B}\subset M_\R$ be the moment
polytope of $-K-B$, let $U=\square^*\subset N_\R$ be the polar dual convex body.
We have $U=\Conv(\frac{e_i}{a_i}|i)$, where $\Delta(1)=\{e_i\}$ and 
$B=\sum_i(1-a_i)V(e_i)$. Since $-K-B$ is $\R$-free, we have
$$
\mld(X,B)=\max\{t>0; N\cap \Int(tU)=\{0\} \}.
$$

Define $\lambda(X,B)$ to be the minimum after all $t>0$ such that 
$M\cap t\cdot (U-U)^*$ contains a basis of $M$. Equivalently,
$\lambda$ is the minimum after all $t>0$ such that $M$ admits a basis
$\varphi_1,\ldots,\varphi_d$ such that the interval $\varphi_i(U)$
has length at most $t$, for all $1\le i\le d$.

\begin{lem}\label{lp} Denote $\lambda=\lambda(X,B)$. The following properties hold:
	\begin{itemize}
		\item[a)] $\sqrt[d]{(-K-B)^d} > \frac{2}{\lambda}$.
		\item[b)] Let $n\ge \lambda$ be an integer. Then the linear system of $\R$-Weil divisors
		$|-nK-nB|$ defines a birational map, and the general member $D\in |-nK-nB|$ satisfies
		the following properties: $n(K+B+\frac{D}{n})\sim 0$ and $\mld(X,B+\frac{D}{n})\ge \frac{1}{n}$.
		\item[c)] If $\dim X$ is fixed and $\lambda(X,B)$ is bounded above, then $X$ belongs to
		finitely many isomorphism types.
	\end{itemize}
\end{lem}

\begin{proof}
	There exists a basis $\varphi_1,\ldots,\varphi_d$ of $M$ such that $\max_i\length \varphi_i(U) \le \lambda$. 
	In particular, $\varphi_i(U) \subset  (-\lambda,\lambda)$ for all $i$. That is, $\pm \varphi_i\in \Int(\lambda\square)$ for all $i$. 
	Let $Q\subset M_\R$ be the convex hull of $\pm \varphi_1,\ldots,\pm \varphi_d$. 
	
	a) We have $\vol_M(Q)=\frac{2^d}{d!}$ and $Q\subset \Int(\lambda \square)$. Since $-K-B$ is nef and big, we obtain
	$$
	(-K-B)^d=d!\vol_M(\square)>\frac{d!\vol_M(Q)}{\lambda^d}=\frac{2^d}{\lambda^d}.
	$$
	
	b) Since $n\ge \lambda$, $\Conv(M\cap \square_{-nK-nB})$
	contains $Q$, which is a $0$-symmetric lattice convex body. Therefore
	$|-nK-nB|$ defines a birational map. We have $n(K+B+\frac{D}{n})\sim 0$. 
	For a toric valuation $E_e\ (e\in N^{prim})$ of $X$, we compute the log discrepancy
	$$
	 n\cdot a_{E_e}(X,B+\frac{D}{n})=-\min_{a\in \{\pm \varphi_1,\ldots,\pm \varphi_d\} }  \langle a,e\rangle  =
	\max_{i=1}^d|\langle \varphi_i,e\rangle|\ge 1.
	$$
	Therefore $\mld(X,B+\frac{D}{n})\ge \frac{1}{n}$.
	
	c) Note that $\Phi=(\varphi_1,\ldots,\varphi_d)\colon N\to \Z^d$ is an isomorphism of lattices, and $\Phi(U)\subset  (-\lambda,\lambda)^d$. Since $\max_{e_i \in \Delta(1)}a_i\le 1$,
	we have $\Delta(1)\subset U$. Therefore $\Phi(\Delta(1))\subset (-\lambda ,\lambda)^d$.
	If $\lambda$ is bounded above, then $\Phi(\Delta(1))$ belongs to a finite collection, hence 
	$X=T_N\emb(\Delta)$ is finite up to isomorphism.
\end{proof}


\subsection{Extending linear maps with length estimates}


	Let $\varphi\colon \Lambda\to \Z$ be a surjective homomorphism, where $\Lambda$ is a lattice of dimension at least $2$. 
	Let $\Lambda_0=\Lambda\cap \varphi^\perp$. Let $C\subset \Lambda_\R$ be a closed convex set such that $\varphi(C)=[p_1,p_2]$, with $p_1<p_2$ real numbers. Let $p=\lambda_1p_1+\lambda_2p_2$, where $\lambda_1,\lambda_2>0$ and $\lambda_1+\lambda_2=1$. Denote $\gamma=\min(\lambda_1,\lambda_2)$ and $w(\varphi)=p_2-p_1$.
	
\begin{prop}\label{lfb}
	Let $\varphi_0\colon \Lambda_0\to \Z$ be a surjective homomorphism.
	Suppose the $\varphi_0$-image of a translation of $C \cap\varphi^{-1}(p)$ which is contained in $\varphi^\perp$
	has finite positive length $w_0$. Then there exists
	a surjective homomorphism $\varphi'\colon \Lambda\to \Z$ such that $\varphi'|_{\Lambda_0}=\varphi_0$
	and 
	$
	\length \varphi'(C)<w(\varphi)+\frac{w_0}{\gamma}.
	$
\end{prop}

\begin{proof}
	Let $\varphi'\colon \Lambda\to \Z$ be an arbitrary extension of $\varphi_0$. Then 
	$\Phi=(\varphi,\varphi')\colon \Lambda\to \Z^2$ is surjective, $\Phi(C)$ is
	a convex body in $\R^2$ such that its projection onto the $x_1$-axis is the interval $[p_1,p_2]$, and its intersection with the line $\{x_1=p\}$ is an interval of length $w_0$. Through each endpoint of the interval $\Phi(C)\cap \{x_1=p\}$, draw a supporting hyperplane to $\Phi(C)$. Together with the inequalities $p_1 \le x_1 
	\le p_2$, these form a quadrilateral or triangle $Q$, which contains $\Phi(C)$. We minimize and maximize the $x_2$-coordinate among the vertices of $Q$. There are two cases.
	
	1) Case both min and max are attained on the same vertical line defining $Q$. Say $(p_2,y)$ and $(p_2,y')$ are such that $[y,y']$ is the projection of $Q$ onto the $x_2$-axis. Then $y'-y=\length  \varphi'(Q)=:d$. Keeping the same intersection $Q\cap \{x_1=p\}$, we slide one of the vertices of $Q\cap\{x_1=p_1\}$ towards the other until $Q$ becomes a triangle $Q'$ with a vertex on the line $x_1=p_1$ and base a segment on the line $x_1=p_2$ of length $d'\ge d$. By similarity, we compute
	$$
	\frac{d'}{w_0}=\frac{p_2-p_1}{p-p_1}=\frac{1}{\lambda_2}.
	$$
	Therefore $d\le d'=\frac{w_0}{\lambda_2}\le \frac{w_0}{\gamma}$. So in this case, $\varphi'$ is the desired extension, satisfying the stronger inequality 
	$
	\length \varphi'(C) \le \frac{w_0}{\gamma}.
	$
	
	2) Case min and max are attained on different vertical lines defining $Q$. Say $(p_1,y)$ and $(p_2,y')$ are vertices of $Q$ such that $[y,y']$ is the projection of $Q$ onto the $x_2$-axis. Let $(p,y'')$ be the vertex of $Q\cap \{x_1=p\}$ with maximal $x_2$-value.	
	Let $\delta=y'-y''$. Keeping the vertex $(p_2,y')$ of $Q$ fixed, we rotate the non-vertical supporting hyperplane of $Q$ until it becomes $\{x_2=y'\}$. The new body $Q'$ has again $[y,y']$ as projection onto the $x_2$-axis, $(p_1,y),(p_1,y')$ are among its vertices, and its intersection with $\{x_1=p\}$ has length $w_0+\delta$. By the argument of case 1), we have 
	$$
	\frac{\length \varphi'(C)}{w_0+\delta}\le \frac{1}{\lambda_1}.
	$$ 	
	For $z\in \Z$, we have an automorphism of $\Z^2$ defined by $x'_1=x_1,x'_2=zx_1+x_2$. It acts on each line $\{x_1=a\}$ as the translation $y\mapsto y+za$.
	After this automorphism, $\delta$ transforms into $\delta+z(p_2-p)$.
	There exists a unique $z\in \Z$ such that $\delta+z(p_2-p)\in [0,p_2-p)$, namely 
	$z=\lceil \frac{-\delta}{p_2-p}\rceil$. Therefore, after changing the second coordinate
	(i.e. changing the lifting of $\varphi_0$ from $\varphi'$ to $\varphi'+z\varphi$), we may suppose $\delta<p_2-p=\lambda_1(p_2-p_1)$. Then
	$
	\length \varphi'(C) <\frac{w_0}{\lambda_1}+p_2-p_1\le \frac{w_0}{\gamma}+w(\varphi).
	$
\end{proof}


\section{A construction}


Let $(X^d,B)$ be a toric weak log Fano with $\mld(X,B)\ge \epsilon>0$. 
The latter assumption is equivalent to 
$$
N\cap \Int(\epsilon U)=\{0\}.
$$

Let $\varphi$ be a primitive element of $M$, that is $\varphi\colon N\to \Z$ is a surjective
linear homomorphism. Let $w$ be the length of the interval $\varphi(U) \subset \R$.

\begin{lem}\label{1w}
$\varphi(U)=[-w_-,w_+]$, where $w_-,w_+\ge 1$ and $w_-+w_+=w$.	
\end{lem}

\begin{proof} Since $U\subset N_\R$ contains $0$ in its interior, so does
	the interval $\varphi(U)\subset \R$. Then $\varphi(U)=[-w_-,w_+]$, where $w_-,w_+>0$ 
	and $w_-+w_+=w$.	

    Since $N_\R=|\Delta|=\cup_{t>0}tU$, there exist $e_i,e_j\in \Delta(1)$ such that 
    $\varphi(e_i)<0<\varphi(e_j)$. Since $e_j\in a_jU$, we obtain $0<\varphi(e_j)\le a_jw_+$.
    Since $\varphi(e_j)$ is a positive integer, $\varphi(e_j)\ge 1$. Since $B$ is effective,
    $a_j\le 1$. Therefore $1\le w_+$. A similar argument gives $1\le w_-$.
\end{proof}

Let $M_0=M/\Z\varphi$ be the quotient lattice, dual to $N_0=N\cap \varphi^\perp$.
Let $\square_0$ be the image of $\square_{-K-B}$ under the projection 
$M_\R\to M_{0,\R}$. The polar dual convex set $U_0=(\square_0)^*$ equals $U\cap \varphi^\perp$.
Since $\epsilon U$ contains the origin in the interior, we deduce
$$
N_0 \cap \Int(\epsilon U_0)=\{0\}.
$$

The surjective linear map $\varphi\colon N\to \Z$ corresponds to a fibration of tori
$T_N\to T_\Z$. Identifying $T_\Z$ with the standard open chart $\{z_0=1\}\subset \bP^1$,
we obtain a toric rational fibration 
$$
\phi \colon X\dashrightarrow  \bP^1, \ \phi (x)=[1:\chi^\varphi(x)].
$$ 
It coincides with the map induced by the invariant linear
system $\Lambda$ defined by the finite set $\{0,\varphi\}\subset M$. The support function
$h\colon N_\R\ni e\mapsto \min(0,\langle \varphi,e \rangle)\in \R$ is $N$-rational
and piecewise linear. If $X'\to X$ is a toric birational contraction, the mobile part 
of $\Lambda_{X'}$ equals $\sum_{e_i\in \Delta_{X'}(1)}-h(e_i)V(e_i)$.
Let $\Delta'$ be the fan consisting of the cones 
$\{ \sigma\cap \varphi^{\le 0},\sigma\cap \varphi^\perp,\sigma\cap \varphi^{\ge 0} | \sigma\in \Delta \}$.
It is a subdivision of $\Delta$, which defines a toric birational contraction 
$\mu\colon X'=T_N\emb(\Delta')\to T_N\emb(\Delta)=X$. The induced map 
$X'\to \bP^1$ is regular, and we obtained a toric resolution of $\phi$:
\[ 
\xymatrix{
	&   X' \ar[dl]_\mu  \ar[dr]^{\phi'} &       \\
	X  \ar@{..>}[rr]^\phi  &   &      \bP^1
}
\]

Let $\mu^*(K+B)=K_{X'}+B_{X'}$ be the log pullback. Then $B_{X'}$ is again invariant,
but possibly not effective in $\mu$-exceptional prime divisors.
 
\begin{lem}\label{bnda}
	$a_E(X,B)\le w$ for every invariant prime divisor $E$ on $X'$.
\end{lem} 

\begin{proof} Suppose $e'\in \Delta'(1)\setminus \Delta(1)$. There exists $\sigma\in \Delta(top)$,
	and $e_i,e_j\in \sigma(1)$ such that $\varphi(e_j)<0<\varphi(e_i)$ and 
	$
	z_{ij}e'=-\langle \varphi,e_j\rangle e_i+ \langle \varphi,e_i \rangle e_j
	$
	for some $z_{ij}\in \Z_{\ge 1}$. There exists $\psi_\sigma\in M_\R$ such that 
	$a_{V(e)}(X,B)=\langle \psi_\sigma,e \rangle$ for every $e\in N^{prim}\cap \sigma$.
	Since $\langle \psi_\sigma,e_i\rangle=a_i$ and $\langle \psi_\sigma,e_j \rangle=a_j$, we obtain 
	$$
	a_{V(e')}(X,B)=\frac{a_ia_j}{z_{ij}}\langle \varphi,\frac{e_i}{a_i}-\frac{e_j}{a_j} \rangle.
	$$
	Since $\frac{e_i}{a_i},\frac{e_j}{a_j}\in U$ and $ \varphi(U)$ has length $w$,
	we obtain $\langle \varphi,\frac{e_i}{a_i}-\frac{e_j}{a_j} \rangle\le w$. Since $a_i,a_j\le 1$,
	we deduce $a_{E_{e'}}(X,B)\le w$.
	
	Suppose $e_i\in \Delta(1)$. Then $a_{V(e_i)}(X,B)=a_i\le 1$. Since $w\ge 2$ by Lemma~\ref{1w},
	we obtain $a_{V(e_i)}(X,B)=a_i<w$. 
\end{proof}

The toric morphism $\phi'\colon X'\to \bP^1$ is a contraction, so its restriction to the open torus
inside $\bP^1$ is a product. Its general fiber is the proper 
toric variety $F=T_{N_0}\emb(\Delta_0)$, where 
$\Delta_0=\{\sigma'\in \Delta' |\sigma'\subset \varphi^\perp\}=\{\sigma\cap \varphi^\perp | \sigma\in \Delta\}$.
Since $-K-B$ is $\R$-free, so is $-K_{X'}-B_{X'}$. By adjunction,
$(K_{X'}+B_{X'})|_F=K_F+B_F$, where $B_F$ is again invariant.
We obtain that $-K_F-B_F$ is $\R$-free, with moment polytope $\square_0$.
From $N_0 \cap \Int(\epsilon U_0)=\{0\}$, we obtain $\mld(F,B_F)\ge \epsilon$.
We conclude that $(F,B_F)$ satisfies the same properties as $(X,B)$, except that 
$B_F$ may not be effective.

Fix $0\le t\le \frac{1}{w}$. Denote $\tilde{B}=(1-t)\Sigma_X+tB$.
Then 
$$
K+\tilde{B}=(1-t)(K+\Sigma_X)+t(K+B)=t(K+B).
$$
Since $\mu^*(K+\Sigma_X)=K_{X'}+\Sigma_{X'}$, we obtain 
$\mu^*(K+\tilde{B})=K_{X'}+\tilde{B}_{X'}$, where 
$$
\tilde{B}_{X'}=(1-t)\Sigma_{X'}+t B_{X'}=\sum_{e'\in \Delta'(1)}(1-ta_{e'})V(e').
$$
By Lemma~\ref{bnda}, $\tilde{B}_{X'}$ is effective.
After adjunction, $\tilde{B}_F=(1-t)\Sigma_F+t B_F$ is again effective. We have 
$$
K_F+\tilde{B}_F=(1-t)(K_F+\Sigma_F)+t(K_F+B_F)=t(K_F+B_F).
$$
Then $\square_{-K_F-\tilde{B}_F}=t \square_0$, $U(F,\tilde{B}_F)=\frac{1}{t} U_0$.
So $-K_F-\tilde{B}_F$ is $\R$-free and $\mld(F,\tilde{B}_F)\ge t\epsilon$.

\begin{lem}\label{indb}
	$\lambda(X,B)<w+\lambda(F,\tilde{B}_F)$.
\end{lem}

\begin{proof} Consider $U\subset N_\R$ and $\varphi\colon N_\R\to \R$.
	Then $\varphi(U)=[-w_-,w_+]$ for some $w_-,w_+>0$ with $w_-+w_+=w$. Let $\gamma$
	be the minimum of the barycentric coordinates of $0\in \varphi(U)$. By Lemma~\ref{1w},
	we estimate 
	$$
	\gamma=\frac{\min(w_-,w_+)}{w}\ge \frac{1}{w}.
	$$
	Consider now $U_0\subset N_0\otimes \R$. Let $\lambda_0=\lambda(N_0,U_0)$.
	That is, there exists a basis $\varphi_1,\ldots,\varphi_{d-1}$ of $M_0=N_0^*$ such
	that $\length \varphi_i(U_0)\le \lambda_0$. By Proposition~\ref{lfb}, each 
	$\varphi_i\colon N_0\to \Z$ admits a lifting $\bar{\varphi}_i\colon N\to \Z$
	such that 
	$$
	\length \bar{\varphi}_i(U)<w+\gamma^{-1}\length \varphi_i(U_0) \le w(1+\lambda_0).
	$$ 
	Since $\varphi,\bar{\varphi}_1,\ldots,\bar{\varphi}_{d-1}$ is a basis of $M$, we obtain 
	$$
	\lambda(N,U)<w(1+\lambda(N_0,U_0)).
	$$
	We have $\lambda(X,B)=\lambda(N,U)$ and $\lambda(N_0,U_0)=t\lambda(F,\tilde{B}_F)$.
	Therefore
	$$
	\lambda(X,B)<w+wt\lambda(F,\tilde{B}_F)\le w+\lambda(F,\tilde{B}_F).
	$$
\end{proof}


\section{Klt complements}


By Lemma~\ref{lp}, suffices to bound $\lambda(X,B)$ from above in Theorem~\ref{kF}.
With the same notations, we show

\begin{prop}\label{mi}
	$\lambda(X,B)\le  \lambda(d,\epsilon)$.
\end{prop}

\begin{proof} Suppose $d=1$.
	Here $X=\bP^1$ and $B=(1-a_+)V(1)+(1-a_-)V(-1)$ with $a_+,a_-\in [0,1]$.
	The assumption $\mld(X,B)\ge \epsilon$ is equivalent to $\min(a_+,a_-)\ge \epsilon$.	
	We have $N=\Z$ and $U=[-\frac{1}{a_-},\frac{1}{a_+}]$. Therefore 
	$$
	\lambda(X,B)=\frac{1}{a_+}+\frac{1}{a_-} \le \frac{2}{\epsilon}=\lambda(1,\epsilon).
	$$
	
	Suppose $d\ge 2$. We have $N\cap \Int(\epsilon U)=\{0\}$ and $\dim U=d$.
	Therefore $\dim (N\cap \Int(\epsilon U))=0<\dim U$. By~\cite[Theorem 5.2.b)]{AI20},
	$w(\epsilon U)\le d(d+1)$. That is, there exists a surjective linear homomorphism
	$\varphi\colon N\to \Z$ such that the interval $\varphi(U)$ has length 
	$$
	w\le \frac{d(d+1)}{\epsilon}.
	$$
	We apply the construction in the previous section to $(X,B)$ and $\varphi$,
	choosing $t=w^{-1}$. Then $\mld(F,\tilde{B}_F)\ge \frac{\epsilon}{w}$.
	By induction, $\lambda(F,\tilde{B}_F)\le \lambda(d-1, \frac{\epsilon}{w})$.
	By Lemma~\ref{indb}, 
	$$
	\lambda(X,B) \le w+\lambda(d-1,\frac{\epsilon}{w}).
	$$
	Note that $\lambda(n,\epsilon)$ is a polynomial in $\epsilon^{-1}$ with positive coefficients.
	Therefore $\lambda(n,\epsilon)$ is decreasing in $\epsilon$. Then
	$$
	w+\lambda(d-1,\frac{\epsilon}{w})\le \frac{d(d+1)}{\epsilon}+\lambda(d-1,\frac{\epsilon^2}{d(d+1)})=
	\lambda(d,\epsilon).
     $$
     Therefore $\lambda(X,B)\le \lambda(d,\epsilon)$.
\end{proof}

\begin{exmp}
	 $\lambda(2,\epsilon)=6\epsilon^{-1}+12\epsilon^{-2}$ and 
	$\lambda(3,\epsilon)=12\epsilon^{-1}+72\epsilon^{-2}+1728 \epsilon^{-4}$.
\end{exmp}


\end{document}